\documentclass[12pt]{amsart}

\usepackage[dvipsnames]{xcolor}
\usepackage{latexsym,amssymb,amsmath,hyperref,amsthm,amsfonts,
caption,subcaption,tikz,comment}
\usepackage{mdwlist,dirtytalk,times,cancel}
\usepackage[capitalise, noabbrev]{cleveref}
\usetikzlibrary{intersections, arrows.meta, automata,er,calc, backgrounds, mindmap,folding, patterns, decorations.markings, fit,decorations, matrix, positioning, shapes.geometric, arrows,through, graphs, graphs.standard}
\usepackage{amsrefs}
\numberwithin{equation}{section}
\usetikzlibrary{positioning}
\usetikzlibrary{arrows}

\newtheorem{theorem}{Theorem}[section]
\newtheorem{lemma}[theorem]{Lemma}
\newtheorem{proposition}[theorem]{Proposition}
\newtheorem{corollary}[theorem]{Corollary}
\newtheorem{conjecture}[theorem]{Conjecture}

\theoremstyle{definition}
\newtheorem{definition}[theorem]{Definition} 
\newtheorem{remark}[theorem]{Remark}
\newtheorem{example}[theorem]{Example}
\newtheorem{acknowledgement}{Acknowledgement}

\newcommand{\K}{\mathbb{K}}

\DeclareMathOperator{\supp}{supp}
\DeclareMathOperator{\Ann}{Ann}

\newcommand{\qand}{\quad \mbox{and} \quad}

\newcommand{\qfor}{\quad \mbox{for} \quad}

\newcommand{\qif}{\quad \mbox{if} \quad}

\newcommand{\Ind}{\text{Ind}}
\newcommand{\st}{\colon}
\newcommand{\h}{\mathcal{H}}

\begin{document}
 
\title[Roller Coaster Gorenstein algebras and Koszul algebras failing the WLP]{Roller Coaster Gorenstein algebras and Koszul algebras failing the weak Lefschetz property}

\author[T. Holleben]{Thiago Holleben}
\address[T. Holleben]
{Department of Mathematics and Statistics,
Dalhousie University
}
\email{hollebenthiago@dal.ca}

\author[L. Nicklasson]{Lisa Nicklasson}
\address[L. Nicklasson]
{Division of Mathematics and Physics,
Mälardalen University
}
\email{lisa.nicklasson@mdu.se}


 
\begin{abstract}
Inspired by the Roller Coaster Theorem from graph theory, we prove the existence of artinian Gorenstein algebras with unconstrained Hilbert series, which we call \emph{Roller Coaster algebras}. Our construction relies on Nagata idealization of quadratic monomial algebras defined by whiskered graphs. The monomial algebras are interesting in their own right, as our results suggest that artinian level algebras defined by quadratic monomial ideals rarely have the weak Lefschetz property. In addition, we discover a large family of G-quadratic Gorenstein algebras failing the weak Lefschetz property.  
\end{abstract}
\maketitle




\section{Introduction}

In 1927, Macaulay proved a famous theorem that classifies sequences that may arise as the Hilbert series of a standard graded algebra in terms of specific growth conditions of the sequence. 
Understanding which numerical sequences can be realized by certain mathematical objects is often a fundamental question in many areas of mathematics. 

An example in commutative algebra is the problem of characterizing Hilbert series of artinian Gorenstein algebras.
Such Hilbert series are typically unimodal, but a first example of an artinian Gorenstein algebra with a nonunimodal 
Hilbert series was discovered by Stanley in 1978 \cite{S1978}. Further on Boij \cite{Bo1995} showed in 1995 that Hilbert series of artinian Gorenstein algebras can have arbitrarily many valleys. Both results rely on the use of {\it Nagata idealizations}.

A famous example in combinatorics is the question of characterizing the integer sequences appearing as independence sequences of graphs. For arbitrary graphs, Alavi et al.~\cite{AMSE1987} showed in 1987 that up to rescaling, the independence sequence of an arbitrary graph can have any shape. In 2000, Brown et al.~\cite{BDN2000} conjectured that the independence  sequence of a well-covered graph is unimodal. This conjecture was disproved shortly after by Michael and Williams in~\cite{MT2003} and led to the following result due to Cutler and Pebody~\cite{CP2017}.

\noindent {\bf Roller Coaster theorem for well-covered graphs} (\cite[Theorem 1.5]{CP2017}).
    {\it Given a positive integer $q$ and a permutation $\pi$ of $\{\lceil q/2 \rceil, \dots, q\}$, there exists a well-covered graph $G$ with independence sequence $i_1, \dots, i_q$ such that
    $$
    i_{\pi(\lceil q/2 \rceil)} < i_{\pi(\lceil q/2 \rceil + 1)} < \dots < i_{\pi(q)}.
    $$
    In other words, the second half of the independence sequence of a well-covered graph is unconstrained.}

One of our main result is the following strengthening of Boij's result~\cite[Theorem 2.7]{Bo1995}, where we combine the graph theoretical techniques from~\cite{CP2017} with Nagata idealization to show the following. 

\begin{theorem}[{\bf Roller Coaster theorem for artinian Gorenstein algebras}]\label{thm:rcintro}
Given a positive integer $d$ and a permutation $\pi$ of the numbers $\{1, \ldots , \lfloor d/2 \rfloor\}$, there exists an artinian Gorenstein algebra with Hilbert series $\sum_{i=0}^dh_it^i$ such that 
$$
h_{\pi(1)} < \dots < h_{\pi(\lfloor d/2 \rfloor)}.
$$
\end{theorem}

The Roller Coaster artinian Gorenstein algebras are in fact \emph{Perazzo algebras}, arising from Nagata idealizations of monomial algebras. Perazzo algebras
have recently been studied in a series of papers \cites{AADF2024, FMM2023, MM2024, MP2024} in the context of Lefschetz properties. Our construction provides a large family of G-quadratic Perazzo algebras failing the weak Lefschetz property.

 \begin{theorem}\label{t:perazzointro}
    Let $H$ be a bipartite graph on $n \geq 12$ vertices and $G$ the whiskering of $H$. Let 
    $$
        F(G) = \sum_{i = 1}^{m} x_i u_{F_i} \in \K[x_1,\dots, x_m, u_1, \dots, u_{2n}],
    $$
    where $F_1, \ldots, F_m$ are the maximal independent sets of $G$ and $u_{F_i} = \prod_{j \in F_i} u_j$. Then $F(G)$ is the Macaulay dual generator of a G-quadratic Perazzo algebra failing the weak Lefschetz property.
\end{theorem}

The absence of the weak Lefschetz property in the Perazzo algebras from~\cref{t:perazzointro} is inherited from their underlying monomial algebras.
These monomial algebras are constructed as artinian reductions of quotients by edge ideals.  
A connection between the Lefschetz properties of edge ideals and invariants of the underlying graphs have been searched for in several recent papers, e.\,g.\ \cites{MM2016, MNS2020, NT2024, T2021}. In the present paper, we focus on edge ideals of whiskered graphs. This is a particularly appealing class, as they define Cohen-Macaulay ideals, and the artinian algebras they define are level. Moreover, whiskered graphs have also been studied from the perspective of piecewise linear (PL) topology, where their independence complexes are known to be simplicial balls or spheres. Our results suggest a connection between the independence number of the graph before whiskering, and the Lefschetz properties. We prove that whiskering a complete graph gives an algebra with the strong Lefschetz property under any monomial artinian reduction. On the other hand, whiskering graphs that are not complete seldom produces algebras satisfying even the weak Lefschetz property. 

\begin{conjecture}\label{c:conjectureintro}
    Let $G$ be a graph with independence number at least $3$ and $I$ the edge ideal of its whiskered graph $H = w(G)$. Then the algebra
    $$
    A(H) = \frac{\K[x_1, \dots, x_{2n}]}{I + (x_1^2, \dots, x_{2n}^2)} 
    $$
    does not have the weak Lefschetz property.
\end{conjecture}

We prove~\cref{c:conjectureintro} in the case where $G$ is bipartite, and as a consequence, for every bipartite graph on $n \geq 12$ vertices we are able to construct an artinian Gorenstein Koszul algebra failing the weak Lefschetz property. Moreover, we note that assuming $G$ does not contain induced triangles, squares or pentagons, it is known~\cite[Corollary 5]{FHN1993} that the algebra $A(G)$ is level if and only if $G$ is whiskered. In particular, if ~\cref{c:conjectureintro} is true, it poses strong restrictions on the weak Lefschetz property of artinian monomial algebras defined by quadratic ideals. 
 
Finally, we note that the results in this paper lead to questions that are almost opposite to previous questions in the literature. In~\cite{MN2013}, Migliore and Nagel conjectured that artinian Gorenstein algebras of socle degree $d \geq 3$ presented by quadrics have the weak Lefschetz property. Although counterexamples for this conjecture were already known (see~\cites{GZ2018,MS2020}),~\cref{c:conjectureintro} together with the results in this paper lead us to the following.

\begin{conjecture}
   Let $d \gg 0$ and $\pi$ be a permutation of the numbers $\{1, \dots, \lfloor d/2 \rfloor\}$. Then there exists an artinian Gorenstein Koszul algebra with Hilbert series $\sum_{i = 0}^d h_i t^i$ such that 
    $$
        h_{\pi(1)} < \dots < h_{\pi(\lfloor d/2 \rfloor)}.
    $$
    In other words, the collection of sequences arising from the first half of the Hilbert series of artinian Gorenstein Koszul algebras of large socle degree is unconstrained.
\end{conjecture}
\section{Background}

\subsection{Simplicial complexes and graphs}

A \emph{simplicial complex} $\Delta$ on vertex set $V$ is a collection of subsets (called \emph{faces}) of $V$ such that if $\tau \subset \sigma \in \Delta$, then $\tau \in \Delta$. Maximal faces (under inclusion) of $\Delta$ are called \emph{facets}. If the facets of $\Delta$ are $F_1, \dots, F_s$, we write $\Delta = \langle F_1, \dots, F_s\rangle$. The \emph{dimension} of a face $\sigma$ of $\Delta$ is defined as $\dim \sigma = |\sigma| - 1$, and the dimension of $\Delta$ is $\dim \Delta = \max(\dim \sigma \st \sigma \in \Delta)$. If every facet of $\Delta$ has the same dimension, we say $\Delta$ is \emph{pure}. 

One dimensional simplicial complexes are called \emph{graphs}. Two vertices $u, v$ in a graph $G$ are said to be \emph{adjacent} if $\{u, v\} \in G$. A set $S$ of vertices of $G$ is said to be \emph{independent} if every pair of vertices in $S$ is not adjacent. Since subsets of independent sets are independent, the collection of all independence sets of $G$ is a simplicial complex $\Ind(G)$ called the \emph{independence complex of $G$}.
The \emph{independence sequence} of $G$ is the sequence $(i_0(G), \dots, i_s(G))$ where $i_j(G)$ denotes the number of independent vertex sets of $G$ of size $j$, where by convention $i_0(G) = 1$. The \emph{independence polynomial} of $G$ is the generating function of the independent set sequence of $G$, in other words, it is the polynomial $i_0(G) + i_1(G) t + \dots + i_s(G) t^s$. The \emph{independence number} of $G$, denoted by $\alpha(G)$, is the size of a maximal independent set of $G$. A graph is said to be \emph{well-covered} if every maximal independent set of $G$ has the same size, and \emph{very well-covered} if $G$ has $2n$ vertices and is a well-covered graph with $\alpha(G) = n$.

Given an arbitrary graph $G$ on vertex set $[n] = \{1, \dots, n\}$, one way to produce a very well-covered graph from $G$ is through a process called \emph{whiskering}. The whiskered graph $w(G)$ is defined as a graph on vertex set $[2n]$, where $\{i, j\} \in w(G)$ if either $\{i, j\} \in G$, or $j = i + n$. We will need the following facts from whiskered graphs. We avoid the definition of a shellable complex since we won't need it throughout this paper.

\begin{theorem}[{\cite[Theorem 4.4]{DE2009}}]\label{thm:whiskeredshellable}
    Let $G$ be a whiskered graph. Then $\Ind(G)$ is shellable, and in particular Cohen-Macaulay.
\end{theorem}
The \emph{edge ideal} $I(G)$ of a graph $G$ on vertex set $[n]$ is an ideal of the polynomial ring $\K[x_1, \dots, x_n]$ over a field $\K$, defined as 
$$
    I(G) = (x_i x_j \ | \ \{i, j\} \in G).
$$
In the case where $H = w(G)$ for some graph $G$ on vertex set $[n]$, we define the edge ideal of $H$ as $I(H) = I(G) + (x_1y_1, \dots, x_n y_n) \in \K[x_1, \dots, x_n, y_1, \dots, y_n]$. That is, the variables $y_1, \ldots y_n$ correspond to the vertices added in when whiskering the graph. 

More generally, given any simplicial complex $\Delta$ on vertex set $[n]$, the ideal 
$$
    I_\Delta = (x_{i_1} \dots x_{i_s} \ | \  \{i_1, \dots, i_s\} \not \in \Delta) \subseteq \K[x_1, \dots, x_n]
$$
is called the \emph{Stanley-Reisner ideal of $\Delta$}. Note that $I(G) = I_{\Ind(G)}$. Simplicial complexes that arise as independence complexes of graphs are called \emph{flag complexes}, and can be characterized as complexes $\Delta$ such that every minimal nonface of $\Delta$ has size $2$.

\subsection{Artinian algebras and the Lefschetz properties} 
A homogeneous ideal $I$  of the polynomial ring $R=\K[x_1, \ldots, x_n]$ defines a graded algebra 
\[
A=R/I = \bigoplus_{i\ge 0} A_i.
\]
The \emph{Hilbert series} of $A$ if the formal power series $\sum_{i\ge 0}h_it^i$ where $h_i=\dim_{\K}A_i$. 
In this paper we study artinian graded algebras, which means that the Hilbert series is a polynomial. In this case we will also refer to the \emph{$h$-vector} $(h_0, h_1, \ldots, h_d)$. 
The \emph{socle} of the algebra $A$ is the vector space $0:(x_1, \ldots, x_n)$, that is the set of $f\in A$ such that $x_if=0$ for all $i$. We note that the last nonzero graded component $A_d$ of an artinian algebra is always contained in the socle. If $A_d$ is the whole socle then the algebra $A$ is called \emph{level}, in which case we refer to the number $d$ as the \emph{socle degree}.  
If the socle is a one dimensional $\K$-space then the algebra is \emph{Gorenstein}. The Hilbert series of an artinian Gorenstein algebra satisfies the symmetry condition $h_{i}=h_{d-i}$. 

A graded algebra $A=R/I$ is said to be \emph{G-quadratic} if the ideal $I$ has a quadratic Gröbner basis. Moreover, $A$ is \emph{Koszul} if the base field $\K$ has a linear free resolution as an $A$-module. By a result in \cite{F1999} every G-quadratic algebra is Koszul. 

An artinian algebra $A=\bigoplus_{1=0}^d A_i$ has the \emph{weak Lefschetz property (WLP)} if there exists a linear form $L\in A_1$ such that all linear maps $\times L:A_i \to A_{i+1}$ defined by multiplication by $L$ have maximal rank. Here having maximal rank means being injective or surjective. Similarly, $A$ has the \emph{strong Lefschetz Property (SLP)} if  there exists a $L\in A_1$ such that all linear maps $\times L^s:A_i \to A_{i+s}$ defined by multiplication by a power $L^s$ have maximal rank. 

If a map $\times L:A_i \to A_{i+1}$ is surjective then so is $\times L:A_j \to A_{j+1}$ for any $j>i$. If $A$ is level, then it also holds that if $\times L:A_i \to A_{i+1}$ is injective then so is $\times L:A_j \to A_{j+1}$ for any $j<i$ (see for example~\cite[Proposition 2.1]{MMN2011}). 

If $A=R/I$ where $I$ is a monomial ideal, then $A$ has the WLP/SLP if and only if the definition is satisfied when taking $L=x_1 + \dots + x_n$ (see for example~\cite[Proposition 2.2]{MMN2011}). 

Given a simplicial complex $\Delta$ on vertex set $[n]$ and an $n$-tuple of positive integers $\bar d = (d_1, \dots, d_n)$, let 
$$
    A(\Delta, \bar d) = \frac{\K[x_1, \dots, x_n]}{I_\Delta + (x_1^{d_1}, \dots, x_n^{d_n})}.
$$
When $\bar d = (2, \dots, 2)$ this construction gives the usual Stanley-Reisner ring, and we simply write $A(\Delta)$. In the case where $\Delta = \Ind(G)$, we use the lighter notation $A(G, \bar d)$ instead of $A(\Ind(G), \bar d)$. Since there are few examples where we consider $\Delta$ a one-dimensional complex, this ambiguity should not cause confusion.
Finally, for a whiskered graph $H = w(G)$ on $2n$ vertices, and $\bar d = (d_1, \dots, d_n)$, we define
$$
    A(H, \bar d) = \frac{\K[x_1, \dots x_n, y_1, \dots, y_n]}{I(H) + (x_1^{d_1}, \dots, x_n^{d_n}, y_1^{d_1}, \dots, y_n^{d_n})}.
$$

Several combinatorial aspects of the complex $\Delta$ can be translated to algebraic properties of the algebra $A(\Delta)$. For example $A(\Delta)$ is level if and only if $\Delta$ is pure, and the Hilbert series of $A(\Delta)$ is given by the polynomial $1 + f_0 t + \dots + f_d t^{d + 1}$, where $f_i$ is the number of faces of $\Delta$ of dimension $i$. In particular, a graph $G$ is well-covered if and only if the algebra $A(G)$ is level, and the Hilbert series of $A(G)$ is given by the independence polynomial of $G$. For a proof of these statements see for example~\cite[Proposition 2.3]{DN2021}.

\section{Lefschetz properties of whiskered graphs}\label{sec:Lefschetz}

In this section we study the Lefschetz properties of artinian algebras $A(w(G),\bar{d})$ associated to whiskered graphs. We will start by considering the whiskering of the complete graph $K_n$. Throughout this section we let $L$ denote the sum of the variables in the corresponding ring.

\begin{lemma}\label{lemma:block_matrix}
    The multiplication maps $\times L^s: A(w(K_n),\bar{d})_i \to A(w(K_n),\bar{d})_{i + s}$ have a block matrix representation
    $$ 
        \begin{bmatrix}
            M_i(n) & 0 \\
            *   & N_i(n)
        \end{bmatrix}  .
    $$
    The block submatrix $M_i(n)$ represents the multiplication map $\times L^s: A_i \to A_{i + s}$, where $A =$ $\K[z_1, \ldots, z_n]/(z_1^{d_1}, \dots, z_n^{d_n})$. The second submatrix 
    $$
        N_i(n) = 
        \begin{bmatrix}
            T_{i - 1}^1(n) & 0 & \dots & 0 \\
            0 & T_{i - 1}^2(n) & 0 & \vdots \\
            \vdots & 0 & \ddots & 0 \\
            0 & \dots & 0 & T_{i - 1}^n(n)
        \end{bmatrix}
    $$
    is a block diagonal matrix where $T_{i}^j(n)$ is the matrix that represents the multiplication map $\times L^s: B_{i} \to B_{i + s}$ where $B = \K[z_1, \dots, z_n]/(z_1^{d_1'}, \dots, z_n^{d_n'})$ with $d_k'=d_k$ if $k \ne j$ and $d_j'=d_j-1.$
\end{lemma}

\begin{proof} 
We can partition the independent sets of $w(K_n)$ in two parts: faces that do not contain any vertex of $K_n$, and faces that contain exactly one vertex of $K_n$. Hence $A(w(K_n),\bar{d})$, considered as $\K$-space, is spanned by monomials not divisible by any $x_i$, and divisible by at most one $x_i$. In fact, we have a direct sum decomposition of the vector space $A(w(K_n), \bar d)_i$ as
\begin{align*}
 \left[ \frac{\K[y_1, \ldots, y_n]}{(y_1^{d_1}, \ldots, y_n^{d_n})} \right]_i & \oplus   x_1  \left[ \frac{\K[x_1, y_2, \ldots, y_n]}{(x_1^{d_1 - 1}, y_2^{d_2}, \ldots, y_n^{d_n})} \right]_{i-1}  \oplus  \cdots \\
 \cdots & \oplus x_i\!\left[ \frac{\K[x_i, y_1, ... \cancel{y_i} ...,  y_n]}{(x_i^{d_i - 1}, y_1^{d_1}, ... \cancel{y_i^{d_i}}..., y_n^{d_n})} \right]_{i-1}  \!\!\!\! \oplus \dots \oplus x_n\!\left[ \frac{\K[x_n, y_1, \ldots, y_{n-1}]}{(x_n^{d_n - 1}, y_1^{d_1}, \ldots, y_{n-1}^{d_{n-1}})} \right]_{i-1}.
\end{align*}
Note that 
\[
\times L^s \left(x_i\!\left[ \frac{\K[x_i, y_1, ... \cancel{y_i} ...,  y_n]}{(x_i^{d_i - 1}, y_1^{d_1}, ... \cancel{y_i^{d_i}}..., y_n^{d_n})} \right]_{i-1}\right) \subseteq x_i\!\left[ \frac{\K[x_i, y_1, ... \cancel{y_i} ...,  y_n]}{(x_i^{d_i - 1}, y_1^{d_1}, ... \cancel{y_i^{d_i}}..., y_n^{d_n})} \right]_{i-1+s} 
\]
as $x_ix_j=0$ and $x_iy_i=0$ in $A(w(K_n),\bar{d})$, which gives the matrix for the multiplication map $\times L^s$ the desired block structure. 
\end{proof}

We can now use the SLP of monomial complete intersections to conclude that the algebra $A(w(K_n),\bar{d})$ has the SLP. 

\begin{proposition}\label{prop:SLP}
The algebra $A=A(w(K_n),\bar{d})$ has the SLP if the characteristic is zero or greater than $D=d_1 + \dots + d_n-n.$
More precisely, in this case the multiplication map $\times L^s: A_i \to A_{i+s}$ is injective when $2i \le D-s$ and surjective otherwise.
\end{proposition}

\begin{proof}
Assume the characteristic of the base field $\K$ is either zero or greater than $D$. 
Recall that a monomial complete intersection $A=\K[z_1, \ldots, z_n]/(z_1^{d_1}, \ldots, z_n^{d_n})$, which has socle degree $D$, has the SLP (\cite{LN2019}, \cite{N2018}, \cite{S1980}). More precisely, the multiplication map $\times L^s:A_i \to A_{i+s}$ is injective for $i \le (D-s)/2$ and surjective for $i \ge (D-s)/2$, where $D= d_1 + \dots + d_n-n$.  Note in particular that the multiplication map is bijective if $D-s$ is even, and $i=(D-s)/2$. 

Now, the block $M_i(n)$ in Lemma \ref{lemma:block_matrix} represents the map $\times L^s: A_i \to A_{i+s}$ where $A=\K[z_1, \ldots, z_n]/(z_1^{d_1}, \ldots, z_n^{d_n})$ which is injective when $i \le (D-s)/2$ and surjective otherwise.

The block $T_{i-1}^j(n)$ in Lemma \ref{lemma:block_matrix}
represents the map 
$\times L^s: B_{i} \to B_{i + s}$
 where 
$B = \K[z_1, \ldots, z_n]/(z_1^{d_1'}, \ldots , z_n^{d_n'})$, with $d_k'=d_k$ for $k \ne j$ and $d_j'=d_j-1$. Note that the algebra $B$ has socle degree $D-1$, so it has the SLP if the characteristic is $p>D$.  
The map $\times L^s: B_{i} \to B_{i + s}$ is then injective when 
\[
i-1 \le \frac{D-1-s}{2} \quad \text{or equivalently} \quad i \le \frac{D-s}{2}+\frac{1}{2}
\]
and surjective when 
\[
i-1 \ge \frac{D-1-s}{2} \quad \text{or equivalently} \quad i \ge \frac{D-s}{2}+\frac{1}{2}.
\]
To argue that the block matrix in Lemma \ref{lemma:block_matrix} has full rank, we need the two blocks to be both injective or both surjective. 
If $D-s$ is even, then both blocks are injective maps for $i\le (D-s)/2$ and surjective maps for $i>(D-s)/2$. If instead $D-s$ is odd, then both blocks are injective maps for $i\le \lfloor(D-s)/2 \rfloor$ and surjective for $i\ge \lceil(D-s)/2 \rceil$. 
\end{proof}

\begin{remark}
    Note that the algebra $A(w(K_n),\bar{d})$ in Proposition \ref{prop:SLP} has the SLP if all the monomial complete intersections appearing as $A$ and $B$, given $\bar{d}$, in Lemma \ref{lemma:block_matrix} have the SLP. For a full classification of monomial complete intersections with the SLP in positive characteristic, see \cite[Theorem 3.4]{N2018}. 
\end{remark}

 In \cref{prop:SLP} we saw that whiskering a graph with independence number 1 produces algebras with the SLP. Next we will see that whiskering a graph with independence number 2 can give rise to algebras with the WLP.

\begin{lemma}\label{lemma:ideal_triag_free}
Let $G$ be a graph with independence number at most $2$. For every non-edge $\{j,k\}$, we have 
    \[
        (I \colon x_j x_k) = (x_i)_{\substack{1 \le i \le n \\ i \ne j,k}} + (x_j^{d_j-1}\!\!,\ x_k^{d_k-1}\!\!,\ y_j,\ y_k,\ y_1^{d_1}, \dots, y_n^{d_n}), \quad \text{where}
    \]
    \[
    I=I(w(G))+(x_1^{d_1}, \ldots, x_n^{d_n}\!,\ y_1^{d_1}, \ldots, y_n^{d_n})
    \]
    for any positive integers $d_1, \ldots, d_n$. 
\end{lemma}

\begin{proof}
We let $\{j,k\}$ be a non-edge of $G$ and note that for every vertex $i$ either $\{i,j\}$ or $\{i,k\}$ is an edge of $G$, as otherwise $\{i,j,k\}$ would be an independent set. In particular, $x_i x_j x_k \in I$ and hence $x_i \in (I \colon x_j x_k)$ for every $i\ne j,k$. Moreover, $y_j, y_k \in (I \colon x_i x_j)$ and we conclude 
    $$
       (I \colon x_j x_k) = (x_i)_{\substack{1 \le i \le n \\ i \ne j,k}} + (x_j^{d_j-1}\!\!,\ x_k^{d_k-1}\!\!,\ y_j,\ y_k,\ y_1^{d_1}, \dots, y_n^{d_n}). \qedhere
    $$
\end{proof}

\begin{proposition}\label{prop:WLP}
    Let $G$ be a graph with independence number at most $2$.  The algebra $A=A(w(G),\bar{d})$ has the WLP if the characteristic is zero and the number $D=d_1+ \dots + d_n-n$ is odd. More precisely, the map $\times L: A_i \to A_{i+1}$ is injective for $i \le (D-1)/2$ and surjective otherwise.  
\end{proposition}

\begin{proof}
 If $G=K_n$ the result holds by~\cref{prop:SLP}.     We proceed by induction on the number of non-edges of $G$.

 Any graph with independence number at most two can be obtained by removing edges one by one from a complete graph $K_n$, keeping the independence number at most two at each step. So, suppose $G$ is obtained from a graph $H$ by removing an edge $\{j,k\}$.  
    
  Let
  \[\overline{I(G)} =I(w(G))+(x_1^{d_1}, \ldots, x_n^{d_n}\!,\ y_1^{d_1}, \ldots, y_n^{d_n})\]
  and $\overline{I(H)}$ analogously, that is $\overline{I(H)} = \overline{I(G)} + (x_j x_k)$.
By Lemma \ref{lemma:ideal_triag_free} we have
\[
J:=(\overline{I(G)}:x_jx_k)=(x_i)_{\substack{1 \le i \le n \\ i \ne j,k}} + (x_j^{d_j-1}\!\!,\ x_k^{d_k-1}\!\!,\ y_j,\ y_k,\ y_1^{d_1}, \dots, y_n^{d_n}).
\]    
According to \cite[Theorem 1]{GLN2022} the map $\times L : A_i \to A_{i+1}$ is injective (surjective) if both 
\[
\times L: [R/\overline{I(H)}]_i \to [R/\overline{I(H)}]_{i+1} \quad \text{and} \quad \times L: [R/J]_{i-2} \to [R/J]_{i-1} 
\]
are injective (surjective). The map $\times L: [R/\overline{I(H)}]_i \to [R/\overline{I(H)}]_{i+1}$ is injective for $i \le (D-1)/2$ and surjective for larger $i$ by the inductive assumption. 
Since $R/J$ is a monomial complete intersection of socle degree $D - 2$, we know the maps $\times L: [R/J]_{i - 2} \to [R/J]_{i - 1}$ are injective whenever $i - 2 \leq (D - 3)/2$ and surjective when $i - 2 \geq  (D - 3)/2$ (one map is bijective). In particular, they are injective for $i \le (D-1)/2$ and surjective for $i \ge (D+1)/2$.
\end{proof}

The weak Lefschetz algebras in Proposition \ref{prop:WLP} do not have the SLP in general. 

\begin{example}
Let $G$ be the whiskering of the complete graph $K_5$ with one edge removed. The algebra $A=A(G)$ has the WLP by Proposition \ref{prop:WLP}. Computation shows that the Hilbert series of $A$ is 
\[
1+10t+31t^2+43t^3+28t^4+7t^5
\]
but the multiplication map $\times L^2: A_2 \to A_4$ is not surjective. Hence $A$ does not have the SLP. 
\end{example}

The next two examples shows that the assumptions on the graph and the number $D$ in Proposition \ref{prop:WLP} can not be dropped.

\begin{example}
Let $G$ be the whiskering of $K_5$ with a triangle removed. The algebra $A=A(G)$ has the Hilbert series 
\[
1+10t+33t^2+50t^3+36t^4+10t^5
\]
but the multiplication map $\times L: A_3 \to A_4$ is not surjective. Hence $A$ does not have the WLP. 
\end{example}

\begin{example}
Let $G$ be the whiskering of $K_4$ with one edge removed. We can not apply Proposition \ref{prop:WLP} to the algebra $A=A(G)$ as $D=4$ is even. Indeed, the Hilbert series of $A$ is 
\[
1+8t+19t^2+18t^3+6t^4 
\]
but the map $\times L: A_2 \to A_3$ is not surjective, and hence $A$ fails the WLP. 
\end{example}

Our next goal is to show that graphs with large independence number produce artinian algebras that fail the WLP.

\begin{lemma}\label{l:transpose}
    Let $A = A(G)$ for some graph $G$. The transpose of the map $\times L: A_i \to A_{i+1}$  is given by partial differentiation with respect to $L$.
\end{lemma}

\begin{proof}
Let $[\times L]^\top: A_{i+1} \to A_i$ denote the transpose of the linear map $\times L: A_i \to A_{i+1}$. We recall that $A$, as a $\K$-space, has a basis of squarefree monomials. 
    Given a squarefree monomial $m \in A_{i+1}$, we have 
    $$
        [\times L]^\top(m) = \sum_{x_i \in \supp m} \frac{m}{x_i} = \sum_{i=1}^n \frac{\partial m}{\partial x_i} = \frac{\partial m}{\partial L} . \qedhere
    $$
\end{proof}

\begin{theorem}\label{thm:not_surj}
    Let $G$ be a graph on $n$ vertices, and let $A=A(w(G))$. If $C$ is an independent set of $G$ then the map
    $ \times L:  A_{i- 1} \to  A_i$ where  $i= \lfloor (n+|C|)/2\rfloor $
    is not surjective. 
\end{theorem}

\begin{proof}
    Let $k =n - |C|$. Without loss of generality we may assume that $C = \{x_{k + 1}, \dots, x_n\}$. Let 
    $$
        f = \prod_{i = 1}^{\left\lfloor\frac{k}{2}\right\rfloor}(y_{2i - 1} - y_{2i}) \prod_{j = k + 1}^n (x_j - y_j)
    $$
    and note that $f$ is a nonzero homogeneous polynomial of degree $\lfloor (n+|C|)/2 \rfloor$. 
    Moreover
        \begin{align*}
            \frac{\partial f}{\partial x_i} &= -\frac{\partial f}{\partial y_i} & & \qfor k + 1 \leq i \leq n,  \\
        \frac{\partial f}{\partial y_i} &= - \frac{\partial f}{\partial y_{i + 1}} & & \qfor  1\leq i < 2  \left\lfloor\frac{k}{2}\right\rfloor,   \\
        \frac{\partial f}{\partial y_k} &= 0 & &  \qif  k \ \text{is odd, and} \\
        \frac{\partial f}{\partial x_i} &= 0 & & \qfor  1\leq i \leq k.
        \end{align*} 
         The properties above imply 
        $$
            \sum_{i=1}^n \frac{\partial f}{\partial x_i} + \sum_{i=1}^n \frac{\partial f}{\partial y_i} = 0.
        $$
    In other words, $f$ is in the kernel of the map $ A_i \to A_{i-1}$ given by differentiation w.\,r.\,t.\, the linear form $L$, where $i=\lfloor (n+|C|)/2 \rfloor$. By Lemma \ref{l:transpose} this is the transpose of the multiplication map $\times L : A_{i-1} \to A_i$, which is then not surjective.
\end{proof}

For Theorem \ref{thm:not_surj} to imply failure of the WLP we need to assure $\dim(A_{i-1}) \ge \dim(A_i)$. In order to do this we use a result from graph theory.

\begin{lemma}\label{l:decrease}
Let $G$ be a whiskered graph on $2n$ vertices and $A = A(G)$. Then $\dim A_{i-1} \ge \dim A_i$ if $3i \ge 2n+2$.
\end{lemma}
\begin{proof}
By \cite[Theorem 2.4]{LM2007} the inequality $\dim A_{i-1} \ge \dim A_i$ holds when 
\[
i-1 \ge \frac{2 \alpha(G)-1}{3} \quad \text{or equivalently} \quad i \ge \frac{2 \alpha(G)+2}{3}
\]
if the graph $G$ is very well covered. As pointed out in \cite{LM2007} whiskered graphs are very well covered, and $\alpha(G)=n$. 
\end{proof}

\begin{corollary}\label{c:everythingfails}
    Let $G$ be a graph on $n$ vertices such that
    $        \alpha(G) \geq n/3 + 2.$    Then the algebra $A=A(w(G))$ does not have the WLP. 
    In particular, the multiplication map $\times L: A_{i-1} \to A_i$ fails surjectivity for $\lceil (2n+2)/3 \rceil \le i \le \lfloor(n+\alpha(G))/2 \rfloor$.
\end{corollary}

\begin{proof}
By Lemma \ref{l:decrease} $\dim A_{i-1} \ge \dim A_i$ holds when $i \ge (2n+2)/3$. 
    Let $C$ be an independent set of $G$ such that $|C| = \alpha(G)$. Then 
    \[
    \left\lfloor \frac{n+|C|}{2} \right\rfloor = \left\lceil \frac{n+|C|-1}{2} \right\rceil \ge \left\lceil \frac{n+\frac{n}{3}+1}{2} \right\rceil =   \left\lceil \frac{4n+3}{6} \right\rceil=  \left\lceil \frac{2n+2}{3} -\frac{1}{6} \right\rceil= \left\lceil \frac{2n+2}{3} \right\rceil  
    \]
so the inequailty $\dim A_{i-1} \ge \dim A_i $ holds for $i = \lfloor(n+|C|)/2 \rfloor$. By Theorem \ref{thm:not_surj} the map $\times L: A_{i-1} \to A_i$ fails to be surjective for all $\lceil (2n+2)/3 \rceil \le i \le \lfloor(n+|C|)/2 \rfloor$.
\end{proof}

\begin{example}\label{ex:Star_graph}
    Let $H_n$ denote the star graph on vertices $\{1, \ldots, n\}$, which has edge set
    \[
    \{ \ \{1,2\}, \ \{1,3\}, \  \ldots, \{1,n\} \ \}.
    \] 
A computation shows that $A(w(H_4))$ does not have the WLP.     
    Since 
    $$
        \alpha(H_n) = n - 1 \geq \frac{n + 6}{3}
    $$
    when $n \ge 5$
    we may apply \cref{c:everythingfails} to conclude $A(w(H_n))$ fails the WLP for all $n \ge 4$. 
    \end{example}

    \begin{example}\label{ex:Broom_graph}
Let $B_m$ be the {\it broom graph} on $m+3$ vertices 
$$
     \{x_1,x_2,x_3,x_4,\ldots,x_{m},y_1,y_2,y_3\}
$$
and edge set
$$ \{\{y_1,y_2\},\{y_2,y_3\}\} \cup \{\{y_3,x_i\}~
|~ i= {1,\ldots,m}\}.$$
In \cite{CFHNVT2023} it is proved by computation that the algebra $A=A(w(B_m))$ fails the WLP for small values of $m$. As 
\[
\alpha(B_m) = m+1 \ge \frac{(m+3)+6}{3}
\]
    holds for $m \ge 3$ we can conclude by \cref{c:everythingfails} that $A$ always fails the WLP for all $m$. 
\end{example}

Computational evidence suggests that having independence number at least three is in fact enough to ensure failure of the WLP. 

\begin{conjecture}\label{conj:notWLP}
If a graph $G$ has independence number $\alpha(G) \ge 3$ then the algebra $A(w(G))$ does not have the WLP. 
\end{conjecture}

\begin{remark}\label{rmk:bipartite_notWLP}
     Conjecture \ref{conj:notWLP} holds for all bipartite graphs $G$. The independence number of a bipartite graph on $n$ vertices is at least $n/2$. If $n \ge 12$ then this implies the inequality on $\alpha(G)$ in Corollary \ref{c:everythingfails}, and consequently the algebra $A(G)$ does not have WLP. 
     For all smaller bipartite graphs the conjecture has been verified computationally. In addition, computational evidence proves that Conjecture \ref{conj:notWLP} holds for all graphs on at most 7 vertices, as well as for a large number of graphs on 8 vertices.   
\end{remark}

The algebras in Corollary \ref{c:everythingfails} fail the WLP due to failure of surjectivity. However, to prove Conjecture \ref{conj:notWLP} injectivity must be taken into account. 

\begin{example}
Let $G$ be the graph obtained by from $K_8$ removing the edges $\{1,2\}$, $\{1,3\}$, $\{2,3\}$ so that $\{1,2,3\}$ becomes an independent set. The Hilbert series of the algebra $A=A(w(G))$ is
\[1+16t+87t^{2}+243t^{3}+400t^{4}+406t^{5}+251t^{6}+87t^{7}+13t^{8}.\]
The multiplication map $\times L: A_4 \to A_5$ fails to be injective, while multiplication by $L$ has maximal rank in all other degrees. 
\end{example}

We recall that every whiskered graph is very well-covered. A characterization of very well-covered graphs on $2n$ vertices with Cohen-Macaulay edge ideal is given in \cite{CRT2011}. In particular the vertex set can be partitioned into two parts $X=\{x_1, \ldots, x_n\}$ and $Y=\{y_1, \dots, y_n\}$ where  $X$ is a minimal vertex cover and $Y$ is a maximal independent set. In the case of a whiskered graph $w(G)$ the minimal vertex cover $X$ is the vertex set of $G$. One may ask if Conjecture \ref{conj:notWLP} can be generalized to very well-covered graphs with Cohen-Macaulay edge ideal, replacing $G$ by the induced subgraph on $X$. However, the conjecture does not hold in this more general setting, as the next example shows. 

\begin{example}
Let $G$ be the graph illustrated below. 

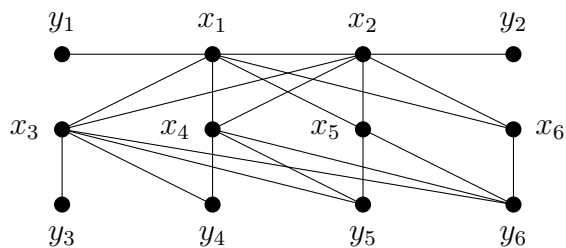
\begin{center}
\begin{tikzpicture}[scale=1]
 
\fill[fill=black,draw=black] (0,0) circle (.1)
node[label=below:$y_3$] {};
\fill[fill=black,draw=black] (2,0) circle (.1)
node[label=below:$y_4$] {};
\fill[fill=black,draw=black] (4,0) circle (.1)
node[label=below:$y_5$] {};
\fill[fill=black,draw=black] (6,0) circle (.1)
node[label=below:$y_6$] {};

\fill[fill=black,draw=black] (0,1) circle (.1)
node[label=left:$x_3$] {};
\fill[fill=black,draw=black] (2,1) circle (.1)
node[label=left:$x_4$] {};
\fill[fill=black,draw=black] (4,1) circle (.1)
node[label=left:{$x_5$}] {};
\fill[fill=black,draw=black] (6,1) circle (.1)
node[label=right:{$x_6$}] {};
\fill[fill=black,draw=black] (2,2) circle (.1)
node[label=above:$x_1$] {};
\fill[fill=black,draw=black] (4,2) circle (.1)
node[label=above:$x_2$] {};
\fill[fill=black,draw=black] (0,2) circle (.1)
node[label=above:$y_1$] {};
\fill[fill=black,draw=black] (6,2) circle (.1)
node[label=above:$y_2$] {};

\draw (0,0) -- (0,1);
\draw (2,0) -- (2,1);
\draw (4,0) -- (4,1);
\draw (6,0) -- (6,1);
\draw (2,0) -- (0,1);
\draw (4,0) -- (0,1);
\draw (6,0) -- (0,1);
\draw (4,0) -- (2,1);
\draw (6,0) -- (2,1);
\draw (6,0) -- (4,1);

\draw (2,2) -- (4,2);
\draw (2,2) -- (0,1);
\draw (2,2) -- (2,1);
\draw (2,2) -- (4,1);
\draw (2,2) -- (6,1);

\draw (4,2) -- (0,1);
\draw (4,2) -- (2,1);
\draw (4,2) -- (4,1);
\draw (4,2) -- (6,1);

\draw (0,2) -- (2,2);
\draw (4,2)-- (6,2);

\end{tikzpicture}
\end{center}
%
    This is a very well-covered graph satisfying the conditions in~\cite{CRT2011} implying that the edge ideal
    \begin{align*}
(&x_{1}x_{2},\,x_{1}x_{3},\,x_{1}x_{4},\,x_{1}x_{5},\,x_{1}x_{6},\,x_{2}x_{3},\,x_{2}x_{4},\,x_{2}x_{5},\,x_{2}x_{6},\,x_{1}y_{1},\,x_{2}y_{2},\,x_{3}y_{3},\,x_{3}y_{4},\, x_{3}y_{5},\, x_{3}y_{6},  \\
    &\quad x_{4}y_{4},\,x_{4}y_{5},\,x_{4}y_{6},\,x_{5}y_{5},\,x_{6}y_{6},\,x_{5}y_{6})
    \end{align*}
    is Cohen-Macaulay.
    The induced subgraph on the vertices $x_1, \ldots, x_6$ has independence number 4 and the artinian algebra $A(G)$ has the WLP.  
\end{example}

\begin{remark}
    Using~\cite[Theorem 1.1, 1.2]{DN2021} and~\cite[Corollary 4.1]{H2024}, the same arguments as in~\cite[Corollary 3.3]{CFHNVT2023} imply the rank in odd characteristics of the first and last multiplication  maps by the sum of variables of algebras of the form $A(\Delta)$ where $\Delta$ is a pseudomanifold with boundary is equal to the rank in characteristic $0$. In particular, a direct consequence of the results mentioned above and the main results in~\cite{CFHNVT2025} characterize when do these maps have full rank assuming $\Delta$ is the independence complex of a very well-covered graph, or a grafted complex (see~\cite[Definition 7.1]{Fa2005}).
\end{remark}

\section{Simplicial Perazzo forms}\label{sec:Perazzo}
In this section we apply the results from Section \ref{sec:Lefschetz} to construct a family of Koszul Gorenstein artinian algebras failing the WLP.

\begin{definition}[{\bf Perazzo forms}]
    A homogeneous polynomial $F \in \K[x_1, \dots, x_n, u_1, \dots, u_m]$ of degree $d$ is called a \emph{Perazzo form} if 
    $$
        F = x_1 g_1 + \dots + x_n g_n + G \quad \text{where} \ g_1, \ldots, g_n, G \in \K[ u_1, \dots, u_m]
    $$
    and the forms $g_1, \ldots, g_n$ are algebraically dependent and linearly independent. 
\end{definition}%
Note that there is no requirement on the form $G$, and throughout this text we take $G= 0$.

Given a set of monomials $M = \{m_1, \dots, m_s\} \subset \K[y_1, \ldots, y_n]$, the \emph{log matrix} of $M$ is the $s \times n$ integer matrix whose $i$-th row is the exponent vector of $m_i$. We recall the following fact about monomials and log matrices. For a proof of~\cref{lemma:log-matrix} below see e.\,g.\ \cite[Lemma 4.2]{S1996}.

\begin{lemma}\label{lemma:log-matrix}
    Monomials $m_1, \dots, m_s$ are algebraically dependent if and only if the rank of their log matrix is less than $s$. In particular, if $s$ is bigger than the number of variables, the monomials are algebraically dependent.
\end{lemma}

Given a subset $F \subseteq \{1, \ldots, m\}$ let $u_{F}$ denote the monomial $\prod_{j \in F} u_j \in \K[u_1, \ldots, u_m]$. For a pure simplicial complex $\Delta = \langle F_1, \dots, F_s \rangle$ we introduce the \emph{simplicial form}
\begin{equation}\label{eq:1}
F(\Delta) = \sum_{i = 1}^s x_i u_{F_i}  \in \K[x_1, \ldots, x_s, u_1, \ldots, u_m].
\end{equation}

\begin{lemma}[{\bf Simplicial Perazzo forms}]\label{lemma:Simplicial-Perazzo}
    Let $\Delta = \langle F_1, \dots, F_s \rangle$ be a pure $d$-dimensional simplicial complex such that the multiplication map $\times L^d: A(\Delta)_1 \to A(\Delta)_{d + 1}$ is not surjective. Then $F(\Delta)$ is a Perazzo form, which we call
    the \emph{ simplicial Perazzo form} of $\Delta$.
\end{lemma}

\begin{proof}
By \cite[Theorem 29]{H2024mm} the multiplication map $\times L: A(\Delta)_1 \to A(\Delta)_{d + 1}$ is exactly the log-matrix of $u_{F_1}, \dots, u_{F_s}$. So if the map is not surjective, then by Lemma \ref{lemma:log-matrix} the monomials $u_{F_1}, \dots, u_{F_s}$ are algebraically dependent, which makes form $F(\Delta)$ Perazzo. 
\end{proof}

For a graph $G$ with a pure independence complex we use the short notation $F(G)$ for the simplicial form $F(\Ind(G))$. 

Now, these forms can be used to produce artinian Gorenstein algebras via Macaulay's Inverse System. We briefly recall the general construction here. Given two polynomial rings $R=\K[x_1, \ldots x_n]$ and $S=\K[X_1, \ldots, X_n]$, we define an action $\circ$ of $R$ on $S$ by taking the variables of $R$ as differential operators $x_i=\frac{\partial }{\partial X_i}$.  A form $F \in S$ (not necessarily Perazzo) defines an ideal of $R$
\[
\Ann(F) = \{g \in R \ | \ g \circ F=0 \} 
\]
called the \emph{annihilator ideal of $F$}. A theorem by Macaulay states that a standard graded artinian algebra is Gorenstein if and only if it is isomorphic to a quotient ring $A_F=R/\Ann(F)$ for some form $F \in S$. We say that $F$ is the \emph{Macaulay dual generator} of $A_F$. The socle degree of the artinian Gorenstein algebra $A_F$ equals the degree of $F$. When there is no risk of confusion the polynomial rings $R$ and $S$ may be identified, as they are indeed isomorphic. 

For the next results we need a construction called Nagata idealization. Let $\K$ be a field of characteristic zero and $R$ a standard graded level artinian $\K$-algebra with socle degree $d$ and graded canonical module $\omega$. The \emph{Nagata idealization} $\tilde R$ of $R$ is the algebra $R \ltimes \omega(-d)$, and the underlying $R$-module is $R \oplus \omega(-d)$. Multiplication in $\tilde R$ is given by 
$$
    (a, x) \cdot (b, y) = (ab, ax + by) \qfor a,b \in R \qand x,y \in \omega(-d).
$$
In particular, we may view $R$ as a subring of $\tilde R$ by identifying $(a, 0) \in \tilde R$ with $a \in R$. For more details on this construction see~\cite[Section 4]{S1978} and~\cite[Section 3]{MSS2021}.

The algebras arising from simplicial forms via Macaulay duality have been considered for example in~\cite{GZ2018,DV2023}. In particular, in~\cite{DV2023} D'Ali and Venturello show the following.

\begin{theorem}[{\cite[Corollary 8.2, Proposition 8.3]{DV2023}}]\label{thm:cmkoszul}
    Let $\Delta$ be a pure simplicial complex of dimension $d$, $F(\Delta)$ the polynomial in~\eqref{eq:1} and $A$ the artinian Gorenstein algebra with Macaulay dual generator $F(\Delta)$. Then the following hold
    \begin{enumerate}
        \item The algebra $A$ is Koszul if and only if $\Delta$ is Cohen-Macaulay, and
        \item The annihilator ideal $\Ann(F(\Delta))$ has a quadratic Gröbner basis if and only if $\Delta$ is shellable.
    \end{enumerate}
    In particular, if $G$ is a whiskered graph, the algebra $A_{F(G)}$ is G-quadratic.
\end{theorem}

\begin{proof}
    The first two equivalences are exactly the statements of~\cite[Corollary 8.2, Proposition 8.3]{DV2023}. The final statement follows directly from~\cref{thm:whiskeredshellable}.
\end{proof}

\begin{lemma}[{\cite[Theorem 2.77]{lefschetzbook} and \cite[Theorem 3.2(4)]{GZ2018}}]\label{lemma:dimension_sum}
Let $A$ be an artinian Gorenstein $\K$-algebra with Macaulay dual generator a simplicial Perazzo form $F(\Delta)$. Then $A$ is the Nagata idealization of the algebra $A(\Delta)$ with its canonical module. In particular, $A$ has socle degree $d=\dim \Delta+2$ and 
\[
\dim_{\K} A_i = \dim_{\K} A(\Delta)_i + \dim_{\K} A(\Delta)_{d-i}.
\]
\end{lemma}

The WLP of algebras defined by Perazzo forms are investigated in \cite{AADF2024, MP2024, MM2024}. 
When the number of variables is small, the presence of WLP for a Perazzo algebra is determined by its Hilbert series.
To fully classify Perazzo algebras with the WLP is an open problem. 

We will now use our results from Section \ref{sec:Lefschetz} to produce Perazzo algebras which fail the WLP.  

\begin{theorem}\label{thm:Perazzo_no_WLP}
    Let $H$ be a graph on $n$ vertices such that $\alpha(H) \geq n/3 + 2$, and let $G=w(H)$. 
    Then the form $F(G)$ is the Macaulay dual generator of a G-quadratic artinian Gorenstein algebra that fails the WLP. 
\end{theorem}

\begin{proof}
Let $A$ denote the artinian Gorenstein algebra defined by $F(G)$. The ideal $\Ann(F(G))$ has a quadratic Gröbner basis by~\cref{thm:cmkoszul}.  
     If the Hilbert series of $A$ is not unimodal then $A$ can not have the WLP. Assuming unimodality we have $\dim A_{i - 1}  \geq \dim A_{i}$ for $i \ge (n+2)/2$, as $A$ has socle degree $n+1$. So for $A$ to have the WLP we need the multiplication map $\times L: A_{i-1} \to A_i$ to be surjective for all such $i$. In particular we would have surjectivity for $i= \lceil (2n+2)/3 \rceil$, as 
    \[
    \frac{2n+2}{3} \ge \frac{n+2}{2} \quad \text{holds for } n \ge 2.
    \] 
     But this would imply surjectivity of the restriction $\times L': A(G)_{i-1} \to A(G)_{i}$, which is not the case by \cref{c:everythingfails}. We can conclude that $A$ does not have the WLP.     
\end{proof}

\begin{corollary}
    Let $H$ be a bipartite graph on $n \geq 12$ vertices and $G$ the whiskering of $H$.
    Then $F(G)$ is a simplicial Perazzo form defining a G-quadratic algebra that fails the WLP.
\end{corollary}

\begin{proof}
    We only have to show that the $F(G)$ is indeed a simplicial Perazzo form. This follows by~\cite[Corollary 5.13]{H2024}, where it is shown that the number of facets of $\Ind(G)$ is bigger than the number of vertices of $G$, as this makes surjectivity of any linear map $A(G)_1 \to A(G)_{n}$ impossible. The failure of WLP follows from Theorem \ref{thm:Perazzo_no_WLP} and Remark \ref{rmk:bipartite_notWLP}.
\end{proof}

\begin{example}\label{ex:Perazzo_Star}
Let $H_n$ denote the Star graph defined in Example \ref{ex:Star_graph}, and let $G_n=w(H_n)$. We saw that $\alpha(H_n) \ge n/3+2$ holds for  $n \ge 5$, so applying \cref{thm:Perazzo_no_WLP,lemma:Simplicial-Perazzo} produces a Gorenstein algebra without the WLP for every $n \ge 5$.
We shall also verify that $\dim(A(G_n)_1) < \dim(A(G_n)_n)$, which then by Lemma \ref{lemma:Simplicial-Perazzo} shows that this is a Perazzo algebra. Note that $\dim(A(G_n)_n)$ equals the number of independent vertex sets in $H_n$, as any independent set of $H_n$ extends to a maximal independent set of $G_n$. The independent sets of $H_n$ are the emptyset, the single central vertex and every subset of the other vertices, so we can conclude that
\[
\dim(A(G_n)_1) = 2n < 2^{n-1}+1= \dim(A(G_n)_n).
\]
\end{example}

\begin{example}\label{ex:Perazzo_Broom}
    Similarly to Example \ref{ex:Perazzo_Star} we can consider the Broom graphs $B_m$ with $n=m+3 \ge 6$ from Example \ref{ex:Broom_graph}.  Here we get 
\[
\dim(A(G)_1) = 2n < 2^{n-2}+2^{n-3}+2= \dim(A(G)_n)
\]
where $G=w(B_m)$,
and we can apply~\cref{thm:Perazzo_no_WLP,lemma:Simplicial-Perazzo} to produce a Perazzo algebra that fails the WLP. 
\end{example}

\section{Roller Coaster Gorenstein algebras}

Although the $h$-vectors of Gorenstein algebras are symmetric, they are not necessarily unimodal. The first example was provided by Stanley \cite{S1978}, and many more examples have been found since then. Boij \cite{Bo1995} proved that there exists artinian Gorenstein algebras having Hilbert functions with arbitrarily many valleys. In this section we use results from graph theory to prove the existence of \emph{Roller Coaster algebras}: artinian Gorenstein algebras whose collection of $h$-vectors is unconstrained in the sense of Definition \ref{def:unconst_seq}.

\begin{definition}\label{def:unconst_seq}
    A collection $\h$ of sequences of length $q$ is said to be \emph{unconstrained} if for every permutation $\pi$ of the numbers $\{1, \ldots, q\}$ there exists a sequence ${\bf a} \in \h$ such that  
    $$
        a_{\pi(1)} < a_{\pi(2)} < \dots < a_{\pi(q)}.
    $$
    We say that a collection of polynomials of a given degree is unconstrained if the coefficient sequences are unconstrained. 
\end{definition}

 It was proved in \cite{AMSE1987} that the collection of independence polynomials of all graphs with a given independence number is unconstrained.    
In fact, a version of this statement holds even if we restrict to well-covered graphs. This was first conjectured in \cite{MT2003}, and later proved in \cite{CP2017}.

\begin{theorem}[Roller Coaster Theorem]
Let $\sum_{i=0}^\alpha h_it^i$ be the independence polynomial of a well-covered graph. Then  $h_i \leq h_{\alpha - i}$ for $i < \alpha/2$. Further, the collection of subsequences $\{h_i\}_{\alpha/2\le i \le \alpha}$ given by all well-covered graphs of independence number $\alpha$ is unconstrained. 
\end{theorem}

We will use ideas from the proof of the Roller Coaster Theorem in \cite{CP2017} to prove the existence of Roller Coaster Gorenstein algebras.

\begin{definition}[\cite{CP2017}]
        A polynomial $a_q x^q + \dots + a_1 x$ is an \emph{approximate well-covered independence polynomial} if for every $\varepsilon > 0$, there exists a well-covered graph $G$ of independence number $q$ and a real number $T$ such that for all $1 \leq k \leq q$,
        $$
            \Bigg{|} \frac{i_k(G)}{T} - a_k \Bigg{|} < \varepsilon
        $$
        where $i_k(G)$ is the number of independent sets of size $k$ of $G$.
        Given such real numbers $T, \varepsilon$ and a graph $G$, we say $G$ is an \emph{ $\varepsilon$-certificate} for $a_q x^q + \dots + a_1 x$ with \emph{ scaling factor} $T$.
\end{definition}

\begin{theorem}[\cite{CP2017}]\label{thm:characterization-approximate}
    For any sequence of non-negative real numbers $(a_1, \dots, a_q)$ satisfying
    $$
        \frac{a_1}{\binom{q}{1}} \leq \frac{a_2}{\binom{q}{2}} \leq \dots \leq \frac{a_q}{\binom{q}{q}},
    $$
    $a_1 x + a_2 x^2 + \dots + a_q x^q$ is an approximate well-covered independence polynomial.
    
\end{theorem}

In order to transfer the graph theoretic results to our algebraic setting, we need the following two lemma. 

\begin{lemma}\label{lemma:inequalities-approximate}
    Let $\sum_{i=1}^qa_ix^i$ be an approximate well-covered polynomial with $\varepsilon$-certificate  $G$. If 
    $$
        a_k + a_{q - k + 1} < a_{\ell} + a_{q - \ell+1} \quad \text{then}
    $$
    $$
        i_k(G) + i_{q - k + 1}(G) < i_{\ell}(G) + i_{q - \ell+1}(G)
    $$
   provided that $\varepsilon>0$ is small enough.  
\end{lemma}

\begin{proof}
    Choose $\varepsilon>0$ such that 
    $$
        \varepsilon < \frac{1}{4}\min (|a_{\ell} + a_{q-\ell+1} - a_k - a_{q - k + 1}| \st 1 \leq k<\ell \leq q).
    $$
    and let $G$ be an $\varepsilon$-certificate of $a_q x^q + \dots + a_1 x$ with scaling factor $T$. Then by definition we know 
    $$
        T(a_k - \varepsilon) < i_k < T(a_k + \varepsilon) \qfor 1 \leq k \leq q. 
    $$
    In particular, adding two of the equations above (for $k$ and $q - k + 1$) gives us 
    $$
        T(a_k + a_{q - k + 1} - 2\varepsilon) < i_k(G) + i_{q - k + 1}(G) < T(a_k + a_{q - k + 1} + 2 \varepsilon) \qfor 1 \leq k \leq q.
    $$
    Now we have two cases:

    \begin{enumerate}
        \item If $a_k + a_{q - k + 1} - a_{\ell} - a_{q - \ell+1} > 0$, then we have 
        \begin{align*}
           0<& \  T(a_k + a_{q - k + 1} - a_{\ell} - a_{q - \ell+1} - 4\varepsilon) \\ 
           <& \  T(a_k + a_{q - k + 1} - 2\varepsilon) - i_{\ell}(G) - i_{q - \ell+1}(G) \\
           < \ & i_k(G) + i_{q - k + 1}(G) - i_{\ell}(G) - i_{q - \ell+1}(G)
        \end{align*}
        where the first inequality holds by our choice of $\varepsilon$, and since $T$ is positive.
        \item Similarly, if $a_k + a_{q - k + 1} - a_{\ell} - a_{q - \ell+1} < 0$, then we have
        \begin{align*}
            & i_k(G) + i_{q - k + 1}(G) - i_{k + 1}(G) - i_{q - k}(G) < \\
            & T(a_k + a_{q - k + 1} + 2\varepsilon) - i_{\ell}(G) - i_{q - \ell+1}(G) < \\
            &T(a_k + a_{q - k + 1} - a_{\ell} - a_{q - \ell+1} + 4\varepsilon) < 0 \qedhere
        \end{align*}
    \end{enumerate}
\end{proof} 

\begin{lemma}\label{lemma:approximate-sequence}
    Let $\pi$ be a permutation of the numbers $\{\lceil \frac{q}{2} \rceil, \dots, q\}$, for some positive integer $q$,  and let
    $$
        a_i = \begin{cases}
            \binom{q}{i} \qfor 1 \leq i \leq \lceil \frac{q}{2} \rceil - 1 \\ 
            3^q + \pi(i)\binom{q}{\lceil q/2 \rceil +1} \qfor \lceil \frac{q}{2} \rceil \leq i \leq q.
        \end{cases}
    $$
    Then $\sum_{i=1}^qa_ix^i$ is an approximate well-covered polynomial, and the inequality 
    \[ a_k + a_{q - k + 1} < a_{\ell} + a_{q - \ell+1}, \quad \text{where} \ \lceil q/2 \rceil \le k<\ell \le q, \]
       holds if and only if $\pi(k) < \pi(\ell)$. 
\end{lemma}

\begin{proof}
    Let  $c = \binom{q}{\lceil q/2 \rceil +1}$. 
    
We first prove that  $\sum_{i=1}^qa_ix^i$ is an approximate well-covered polynomial.   
     By~\cref{thm:characterization-approximate} this is equivalent to showing 
    \begin{equation}\label{e:inequalities-approximate}
        \frac{a_1}{\binom{q}{1}} \leq \frac{a_2}{\binom{q}{2}} \leq \dots \leq \frac{a_q}{\binom{q}{q}}.
    \end{equation}
    We have 
\[
 \frac{a_i}{\binom{q}{i}} = \begin{cases}
            1 \qfor 1 \leq i \leq \lceil \frac{q}{2} \rceil - 1 \\ 
            \dfrac{3^q + \pi(i)c}{\binom{q}{i}} \qfor \lceil \frac{q}{2} \rceil \leq i \leq q.
        \end{cases}
\]    
    and so we have to prove 
    $$
        \frac{3^q + \pi(k)c}{\binom{q}{k}} \leq \frac{3^q + \pi(k + 1)c}{\binom{q}{k + 1}} \qfor \lceil q/2 \rceil \leq k < q
    $$
    But in this case we have 
    \begin{align*}
&        \frac{3^q + \pi(k)c}{3^q + \pi(k + 1)c} \leq \frac{1 + \frac{qc}{3^q}}{1 + \frac{c}{3^q}} = 1 + \frac{(q-1)c}{3^q + c} = \\
 & \quad =        1 + \frac{q-1}{1 + \frac{3^q}{c}} \leq 1 + \frac{q - 1}{1 + \frac{3^q}{2^q}} \leq 1 + \frac{2^q(q - 1)}{3^q} \leq 1 + \Big{(}\frac{2}{3}\Big{)}^qq 
    \end{align*}
    On the other hand, we also have 
    \begin{equation*}
        \frac{\binom{q}{k}}{\binom{q}{k + 1}} = \frac{k+1}{q-k} \geq \frac{\frac{q}{2} + 1}{\frac{q}{2}} = 1 + \frac{2}{q}.
    \end{equation*}
    Now for $q \geq 10$ we see that 
    $$
         \frac{3^q + \pi(k)c}{3^q + \pi(k + 1)c} \leq 1 + \Big{(}\frac{2}{3}\Big{)}^q q \leq 1 + \frac{2}{q} \leq \frac{\binom{q}{k}}{\binom{q}{k + 1}},
    $$
    and since we can check computationally (for $q < 10$) that
    $$
        \frac{3^q + \pi(k)c}{3^q + \pi(k + 1)c} \leq \frac{1 + \frac{qc}{3^q}}{1 + \frac{c}{3^q}} \leq 1 + \frac{2}{q} \leq \frac{\binom{q}{k}}{\binom{q}{k + 1}}
    $$
    we conclude~\eqref{e:inequalities-approximate} holds for any $q$, and so the claim follows.
    
    Next assume
    \[ a_k + a_{q - k + 1} < a_{\ell} + a_{q - \ell+1}, \quad \text{for some} \ \lceil q/2 \rceil \le k<\ell \le q. \]
    By the definition of $a_i$ this inequality states
    \[
    3^q + \pi(k)c + \binom{q}{k-1} < 3^q + \pi(\ell)c + \binom{q}{\ell-1} 
    \]
    or equivalently 
    \[
    \binom{q}{k-1} - \binom{q}{\ell-1} <c(\pi(\ell)-\pi(k)).
    \]
    Note that the left hand side is positive, and since  $c = \binom{q}{\lceil q/2 \rceil +1}$ the inequality holds precisely when $\pi(\ell)>\pi(k)$. 
\end{proof}

We are now able to show the existence of Roller Coaster Gorenstein algebras, making use of the Perazzo forms from \cref{sec:Perazzo}.

\begin{theorem}\label{t:perazzo-unconstrained}
    The collection $\h$ of sequences $\{h_1, \dots, h_{\lfloor d/2 \rfloor}\}$ where the $h_i$ are entries of Hilbert functions of (flag) simplicial Perazzo algebras of socle degree $d$ is unconstrained for every $d$. In particular, the first half of Hilbert functions of Perazzo algebras is an unconstrained collection of sequences.
\end{theorem}

\begin{proof}
    First note that by symmetry considering the first half of the Hilbert series is equivalent to considering the second half. Let $\pi$ be a permutation of the numbers $\{\lceil d/2 \rceil, \ldots, d-1\}$. We need to show there exists a flag simplicial complex $\Delta$ such that the Nagata idealization of $A(\Delta)$ is a (flag simplicial) Perazzo algebra, and the Hilbert function $h_1, \dots, h_{\lfloor d/2 \rfloor}$ satisfies the desired inequalities. 
    
Let $p=\sum_{i=1}^{d - 1}a_ix^i$ be the approximate well-covered polynomial defined in Lemma \ref{lemma:approximate-sequence}.   Let $G$ be a $\varepsilon$-certificate for $p$ with scaling factor $T$, where we have chosen $\varepsilon>0$ so that $2 \varepsilon<a_{d - 1}-a_1$, and small enough so that we can apply Lemma \ref{lemma:inequalities-approximate}. Then 
    $$
        i_1(G) < T(a_1 + \varepsilon) < T(a_q - \varepsilon) < i_{d - 1}(G)
    $$
    In particular, the independence complex of $G$ has more facets than vertices, and hence it defines a flag simplicial Perazzo algebra with Hilbert function given by  
    \[
    h_k= i_k(G) + i_{d-k}(G),
    \]
where $i_d(G) = 0$. By Lemma \ref{lemma:approximate-sequence} the sequence $a_1, \dots, a_{\lfloor d/2 \rfloor} + a_{d - \lfloor d/2 \rfloor - 1}$ satisfies the inequalities prescribed by $\pi$, and by Lemma \ref{lemma:inequalities-approximate} this transfers to our Hilbert function $h_k$. 
\end{proof}

We have now showed the existence of a \emph{$\pi$-Roller Coaster algebra} in the sense that for any permutation $\pi$ of the numbers  $\{1, \ldots, \lfloor d/2 \rfloor \}$ there exists an artinian Gorenstein algebra whose $h$-vector satisfies 
\[
  h_{\pi(1)} < \dots < h_{\pi(\lfloor d/2 \rfloor)}.
\]

    In~\cite{TY2006} the authors show that if a pure $d$-dimensional simplicial complex on $n$ vertices  
    has at least $\binom{n}{n - d - 1} - 2(n - d - 1) + 1$ facets
    then it is Cohen-Macaulay. It was recently proved in \cite{DNSTW2024} that the same bound actually implies the stronger property that $\Delta$ is vertex decomposable.  
    Due to \cref{thm:cmkoszul}, and since vertex decomposable complexes are shellable (and hence Cohen-Macaulay), one may ask if there is a version of  Theorem \ref{t:perazzo-unconstrained} for Koszul artinian Gorenstein algebras. 

\begin{conjecture}
Given an integer $d \gg 0$ and a permutation $\pi$ of the numbers $\{1, \dots, \lfloor d/2 \rfloor\}$. Then there exists an artinian Gorenstein Koszul algebra with Hilbert series $\sum_{i = 0}^d h_i t^i$ such that 
    $$
        h_{\pi(1)} < \dots < h_{\pi(\lfloor d/2 \rfloor)}.
    $$
    In other words, the collection of sequences arising from the first half of the Hilbert series of artinian Gorenstein Koszul algebras of large socle degree is unconstrained.
\end{conjecture}

\begin{acknowledgement}
    We would like to thank Susan Cooper, Sara Faridi and Adam Van Tuyl for inspiring dicussions on the topics of this paper.
    Part of the work on this project
took place at the Fields Institute in 
Toronto, Canada while the authors
participated in the ``Thematic Program in
Commutative Algebra and Applications''; we
thank Fields for the financial support 
received while working on this project. The first named author would like to thank Yufei Zhang for insightful discussions on the Roller Coaster theorem for graphs. Every computation for this project was done using Macaulay2~\cite{M2}.
\end{acknowledgement}

\bibliographystyle{plain}
\bibliography{bibliography.bib}

\end{document}